\documentclass[12pt,leqno]{amsart}
\usepackage{amssymb,amsthm,amsmath,latexsym}
\newcommand{\vfi}{\varphi}
\newtheorem{theorem}{\sc Theorem}[section]

\newtheorem{lem}[theorem]{\sc Lemma}

\newtheorem{cor}[theorem]{\sc Corollary}

\newtheorem{rem}[theorem]{\sc Remark}

\newtheorem*{thmA}{Theorem A}
\newtheorem*{thmB}{Theorem B}
\newtheorem*{corC}{Corollary C}
\title[Non-abelian tensor square]{Non-abelian 
tensor square of finite-by-nilpotent groups}
\author{R. Bastos }
\address{ Departamento de Matemat\'atica, Universidade de Bras\'ilia,
Brasilia-DF, 70910-900 Brazil }
\email{bastos@mat.unb.br}
\author{N.\,R. Rocco }
\address{ Departamento de Matemat\'atica, Universidade de Bras\'ilia,
Brasilia-DF, 70910-900 Brazil }
\email{norai@unb.br}
\subjclass[2010]{20E34; 20F14; 20F24}
\keywords{Structure theorems; derived series; FC-groups}

\begin{document}

\maketitle

\begin{abstract}
Let $G$ be a group. We denote by $\nu(G)$ an extension of the non-abelian tensor square $G \otimes G$ by $G \times G$. We prove that if $G$ is finite-by-nilpotent, then the non-abelian tensor square $G \otimes G$ is finite-by-nilpotent. Moreover, $\nu(G)$ is nilpotent-by-finite (Theorem~A). Also we characterize BFC-groups in terms of $\nu(G)$ (Theorem~B).
\end{abstract}

\maketitle

\section{Introduction}

Let $G$ and $G^{\varphi}$ be groups, isomorphic via 
$\varphi: g \mapsto g^{\varphi}$ for all $g \in G$. 
Consider the group $\nu(G),$ introduced in \cite{NR1} as 
\begin{equation} \label{eq:presenta0}
 \nu(G) = \left \langle G \cup G^{\varphi} \, | \, [g, h^{\varphi}]^k = 
 [g^k, (h^k)^{\varphi}] = [g, h^{\varphi}]^{k^{\varphi}}, \, \forall g,h,k 
 \in G \right \rangle.   
\end{equation}
It is a well known fact (see \cite{NR1}) that the subgroup 
$\Upsilon(G) = 
[G, G^{\varphi}]$ of $\nu(G)$ is canonically isomorphic with the non-abelian 
tensor square $G \otimes G$, as defined by R. Brown and J.-L. Loday in their seminal 
paper \cite{BL}, the isomorphism being induced by $g \otimes h \mapsto 
[g, h^{\varphi}].$ The normality of $\Upsilon(G)$ in $\nu(G)$ gives the decomposition 
\begin{equation} \label{eq:decomposition}
 \nu(G) = \left ( [G, G^{\varphi}] \cdot G \right ) \cdot G^{\varphi},
\end{equation}
where the dots mean (internal) semidirect products. With this in mind we shall 
write $\nu(G) = \Upsilon(G) \, G \, G^{\varphi}$ and use 
$\Upsilon(G), \; [G, G^{\varphi}]$ or $G \otimes G$ indistinctly to denote 
the non-abelian tensor square of $G.$ 

The group $\nu(G)$ inherits many properties of the argument $G$; for instance, 
if $G$ is a finite $\pi-$group ($\pi$ a set o primes), nilpotent, solvable, 
 polycyclic-by-finite, or a locally finite group, then so is $\nu(G)$ (and 
hence also $G \otimes G$ and $G \wedge G$) \cite{NR1,BJR,LN,BM,M}. For a deeper discussion of commutator approach we refer the reader to \cite{K,NR}.

In the present article we want to study the structure of $G \otimes G$ and $\nu(G)$ when $G$ is a finite-by-nilpotent group. 

An important result in the context of finite-by-nilpotent groups, due to R. Baer \cite[14.5.1]{Rob}, states that if $Z_k(G)$ is a subgroup of finite index in $G$, for some positive integer $k$, then the subgroup $\gamma_{k+1}(G)$ is finite. The converse does not hold in general; however in \cite{Hall}, Ph. Hall obtained that if $\gamma_{k+1}(G)$ is finite, for some positive integer $k$, then $Z_{2k}(G)$ is a subgroup of finite index in $G$. This theorem shows that if $G$ is finite-by-nilpotent, then $G$ is nilpotent-by-finite. In the present paper we establish the following related result.   

\begin{thmA} \label{a}
Let $G$ be a finite-by-nilpotent group. Then  
\begin{itemize}
 \item[$(a)$] The non-abelian tensor square $G \otimes G$ is finite-by-nilpotent;  
 \item[$(b)$] The group $\nu(G)$ is nilpotent-by-finite; 
\end{itemize}
\end{thmA}

The converse of Theorem~A $(a)$ does not hold. In \cite[Theorem 22]{BM} Blyth and Morse proved that $\Upsilon(D_{\infty})$ is abelian, where $D_{\infty} = \langle a,b\mid  a^2=1, a^b = a^{-1}\rangle$. More precisely, $\Upsilon(D_{\infty})$ is isomorphic to $C_2 \times C_2 \times C_2 \times C_{\infty}$.  On the other hand, $\gamma_k(D_{\infty})$ is an infinite cyclic group, for every positive integer $k$. So, in particular, $D_{\infty}$ cannot be finite-by-nilpotent.  

It is an immediate consequence of~\cite[Proposition 9]{BJR} that if $G$ has a central subgroup $Z$ of finite index then also $G \otimes G$ has a central subgroup of finite index. Such a group is called \textit{central-by-finite}. In particular, in this case the center $Z(G)$ is a subgroup of finite index in $G$. I. Schur \cite[10.1.4]{Rob} showed that if $G$ is a central-by-finite group then the derived subgroup $G'$ is finite and $exp(G')$ divides $\vert G/Z(G) \vert$. It is clear that every central-by-finite group is also a BFC-group. Recall that a group $G$ is called a BFC-group if there is a positive integer $d$ such that no element of $G$ has more than $d$ conjugates. B.\,H. Neumann improved Schur's theorem in a certain way, showing that the group $G$ is a BFC-group if and only if the derived subgroup $G'$ is finite \cite{BHN}. The following result give us another characterization of BFC-groups in terms of $\nu(G)$.     

\begin{thmB} \label{b}
Let $G$ be a group in which $G'$ is finitely generated. The following properties are equivalent: 
\begin{itemize}
 \item[$(a)$] $G$ is BFC-group;
 \item[$(b)$] $\nu(G)'$ is central-by-finite;
 \item[$(c)$] $\nu(G)''$ is finite.
\end{itemize} 
\end{thmB}

In Theorem~B, the hypothesis that the derived subgroup $G'$ is finitely generated is essential (see Remark \ref{dihedral}, below). 

It is well know that the converse of Schur's theorem does not hold. \\

\noindent {\bf Definition.} {\it A group $G$ is called $\nu$-group if and only if there exists a group $H$ with derived subgroup $H'$ finitely generated such that $G$ is isomorphic to $\nu(H)'$.} \\

We can now rephrase Theorem~B (b)--(c) as follows. 

\begin{corC} \label{c}
Let $G$ be a $\nu$-group. If $G'$ is finite, then $G$ is central-by-finite.
\end{corC}

For more details concerning groups satisfying the converse of Schur's theorem see \cite{H,N,Y,GK}.

In the next section we collect some results of the non-abelian tensor square and related construction that are later used in the proofs of our main theorems. Section 3 contains the proofs of the main results.

\section{The group $\nu(G)$}

The following basic properties are consequences of 
the defining relations of $\nu(G)$ and the commutator rules (see \cite[Section 2]{NR1} for more details). 

\begin{lem} 
\label{basic.nu}
The following relations hold in $\nu(G)$, for all 
$g, h, x, y \in G$.
\begin{itemize}
\item[$(i)$] $[g, h^{\varphi}]^{[x, y^{\varphi}]} = [g, h^{\varphi}]^{[x, 
y]}$; 
\item[$(ii)$] $[g, h^{\varphi}, x^{\varphi}] = [g, h, x^{\varphi}] = [g, 
h^{\varphi}, x] = [g^{\vfi}, h, x^{\vfi}] = [g^{\vfi}, h^{\vfi}, x] = 
[g^{\vfi}, 
h, x]$;
\item[$(iii)$] If $h \in G'$ (or if $g \in G'$) then $[g, h^{\varphi}][h, 
g^{\varphi}]=1$;
\item[$(iv)$] $[g, [h, x]^{\varphi}]= [[h, x], g^{\varphi}]^{-1}$;
\item[$(v)$] $[[g,h^{\vfi}],[x,y^{\vfi}]] = [[g,h],[x,y]^{\vfi}]$.
\end{itemize}
\end{lem}

In the notation of \cite[Section 2]{NR2}, we have the derived map $$\rho': \Upsilon(G) \to G'$$ given by $[g,h^{\vfi}] \mapsto [g,h]$. Let us denote by $\mu(G)$ the kernel of $\rho'$. In particular, $$\dfrac{\Upsilon(G)}{\mu(G)} \cong G'.$$ 

\begin{rem} \label{rem.quotient}
Let $N$ be a normal subgroup of $G$. We set $\overline{G}$ for the quotient group $G/N$ and the canonical epimorphism $\pi: G \to \overline{G}$ gives raise to an epimorphism $\widetilde{\pi}: \nu(G) \to \nu(\overline{G})$ such that $g \mapsto \overline{g}$, $g^{\varphi} \mapsto \overline{g^{\varphi}}$, where $\overline{G^{\varphi}} = G^{\varphi}/N^{\varphi}$ is identified with $\overline{G}^{\varphi}$. 
\end{rem}

In the following lemma we collect some results which can be found in \cite{NR1} and \cite{NR2}, respectively.

\begin{lem}\label{lem.general} Let $n$ be a positive integer and $G$ be a group. 
\begin{enumerate}
\item[$(i)$] (\cite[Proposition 2.5]{NR1}) With the above notation we have
\begin{itemize}
 \item[$(a)$] $[N,G^{\varphi}] \lhd \nu(G)$, $[G,N^{\varphi}] \lhd \nu(G)$;
 \item[$(b)$] $Ker \ \widetilde{\pi} = \langle N,N^{\varphi}\rangle [N,G^{\varphi}] \cdot [G,N^{\varphi}]$. 
 \end{itemize}
\item[$(ii)$] ({\rm \cite[Theorem 3.1]{NR1}}) For $i \geqslant 2$ the $i$-th term of the lower central series of $\nu(G)$ is given by $$\gamma_i(\nu(G)) = \gamma_i(G) \gamma_i(G^{\varphi}) [\gamma_{i-1}(G),G^{\varphi}] [G,\gamma_{i-1}(G^{\varphi})].$$ 
\item[$(iii)$] ({\rm \cite[Theorem 3.3]{NR1}}) For $i \geqslant 2$ the $i$-th term of the derived series of $\nu(G)$ is given by $$ (\nu(G))^{(i)} = G^{(i)} (G^{\varphi})^{(i)} [G^{(i-1)},(G^{\varphi})^{(i-1)}].$$
 \item[$(iv)$] (\cite[Proposition 2.7 (i)]{NR2}) $\mu(G)$  consists of all elements of $\Upsilon(G)$ of the form $$[h_1,g_1^{\vfi}]^{\epsilon_1} \ldots  [h_s,g_s^{\vfi}]^{\epsilon_s} $$ such that $[h_1,g_1]^{\epsilon_1} \ldots  [h_s,g_s]^{\epsilon_s} = 1$, where $s$ is a natural number, $h_i,g_i \in G$, $\epsilon_i \in \{1,-1\}$, $1 \leqslant i \leqslant s$.
 \item[$(v)$] (\cite[Proposition 2.7 (ii)]{NR2}) $\mu(G)$ is a central subgroup of $\nu(G)$.  
\end{enumerate}
\end{lem}

The following result will be needed in the proof of Theorem~A (b). 

\begin{lem} \label{lem.Z}
Let $n$ be a positive integer and $G$ a group. Then $$[Z_{n}(G), G^{\varphi}][G, Z_{n}(G)^{\varphi}] \leqslant Z_{n}(\nu(G)).$$ 
\end{lem}

\begin{proof}
The case where $n=1$ is covered by \cite[Proposition 2.7]{NR1}. So we will assume that $n \geq 2$. We first prove that $[Z_{n}(G),G^{\vfi}] \leqslant Z_{n}(\nu(G))$. Choose arbitrarily elements $a \in Z_{n}(G)$ and $b \in G$. Since $\nu(G) = \Upsilon(G) \  G \  G^{\vfi}$, it is sufficient to show that $$[[a,b^{\varphi}],x_1^{\vfi}, \ldots,x_{n}^{\vfi}] = 1 = [[a,b^{\varphi}],x_1, \ldots,x_{n}],$$ for every elements $x_1, \ldots,x_{n} \in G$. Let $x_1, \ldots, x_n \in G$. Repeated application of Lemma \ref{lem.general} (ii) enables us to write 

$$ [[a,b^{\varphi}],x_1^{\vfi}, x_2^{\vfi}, \ldots, x_{n-1}^{\vfi}, x_{n}^{\vfi}] = [[a,b,x_1, x_2, \ldots,x_{n-1}],x_{n}^{\vfi}] = 1$$ 

\centering{and}
$$ [[a,b^{\vfi}],x_1,x_2 \ldots,x_{n}] = [[a,b,x_1,x_2, \ldots, x_{n-1}],x_n^{\vfi}] = 1.$$ Further, using only obvious modifications of the above argument one can show that $[G,Z_{n}(G)^{\vfi}] \leqslant Z_{n}(\nu(G))$. This completes the proof.       
\end{proof}

\section{Proofs of the main results}

The following lemma is well known. We supply the proof for the reader's convenience.

\begin{lem} \label{lem.equiv}
Let $G$ be a group. The following properties are equivalent:
\end{lem}
\begin{itemize}
 \item[$(i)$] $G$ is finite-by-nilpotent;
 \item[$(ii)$] There exists a positive integer $k$ such that the subgroup $Z_k(G)$ is a subgroup of finite index in $G$.
\end{itemize}

\begin{proof}
Suppose that $G$ is finite-by-nilpotent. By definition, there exists a positive integer $k$ such that $\gamma_{k+1}(G)$ is finite. By Hall's theorem \cite{Hall}, $Z_{2k}(G)$ is a subgroup of finite index in $G$. Thus $(i)$ implies $(ii)$. 

Finally let $G$ satisfying $(ii)$. According to Baer's theorem \cite[14.5.1]{Rob}, $\gamma_{k+1}(G)$ is finite, which completes the proof.
\end{proof}

\begin{thmA} \label{a}
Let $G$ be a finite-by-nilpotent group. Then  
\begin{itemize}
 \item[$(a)$] The non-abelian tensor square $G \otimes G$ is finite-by-nilpotent;  
 \item[$(b)$] The group $\nu(G)$ is nilpotent-by-finite; 
\end{itemize}
\end{thmA}
\begin{proof}
{\noindent} {$(a)$.} Since $(G \otimes G)/\mu(G)$ is isomorphic to $G'$ and $\mu(G)$ is a central subgroup of $\nu(G)$, it follows that the quotient $(G \otimes G)/Z(G \otimes G)$ is isomorphic to a finite-by-nilpotent group. Therefore $G \otimes G$ is finite-by-nilpotent. \\ 

{\noindent} {$(b)$.} By definition, there exists a positive integer $k$ such that $\gamma_k(G)$ is finite. According to Hall's result \cite{Hall}, $Z_{2k}(G)$ is a subgroup of finite index in $G$. Set $\overline{G} = G/Z_{2k}(G)$. By Remark \ref{rem.quotient}, there exists an epimorphism $\widetilde{\pi} : \nu(G) \to \nu(\overline{G})$. That $\nu(\overline{G})$ is finite follows from \cite[Proposition 2.4]{NR1}. Lemma \ref{lem.general} now shows that $$ Ker \  \widetilde{\pi} = \langle Z_{2k}(G), Z_{2k}(G)^{\varphi} \rangle [Z_{2k}(G), G^{\varphi}][G, Z_{2k}(G)^{\varphi}]. $$  

Since $\nu(\overline{G})$ is finite, it is sufficient to show that $Ker \ \widetilde{\pi}$ is nilpotent. By Lemma \ref{lem.Z}, $[Z_{2k}(G), G^{\varphi}][G, Z_{2k}(G)^{\varphi}] \leqslant Z_{2k}(\nu(G))$. On the other hand, $\nu(Z_{2k}(G)) \geqslant \langle Z_{2k}(G), Z_{2k}(G)^{\varphi} \rangle$ is nilpotent \cite[Corollary 3.2]{NR1}. Hence $Ker \ \widetilde{\pi}$ is nilpotent.
\end{proof}

Now we will deal with Theorem~B: Let $G$ be a group in which $G'$ is finitely generated. The following properties are equivalent: 

\begin{itemize}
 \item[$(a)$] $G$ is BFC-group;
 \item[$(b)$] $\nu(G)'$ is central-by-finite;
 \item[$(c)$] $\nu(G)''$ is finite.
\end{itemize}

It is well known that the finiteness of the non-abelian tensor square $G \otimes G$, does not imply that $G$ is a finite group. For instance, the Pr\"uffer group $\mathbb{Z}(p^{\infty})$ is an example of an infinite group such that $\mathbb{Z}(p^{\infty}) \otimes \mathbb{Z}(p^{\infty}) = 0$ (and so, finite). This is the case for all torsion, divisible abelian group. A useful result, due to Niroomand and Parvizi \cite{NP}, provides a sufficient condition for a group to be finite.

\begin{lem} \label{fg}
Let $G$ be a finitely generated group in which the non-abelian tensor square $[G,G^{\varphi}]$ is finite. Then $G$ is finite. 
\end{lem}

We are now in a position to prove Theorem~B.

\begin{proof}[\bf Proof of Theorem B] {\noindent} {$(a)$ implies $(b)$.} By Lemma \ref{lem.general} (v), $\mu(G) \leqslant Z(\nu(G))$. Moreover, the quotient $\Upsilon(G)/\mu(G) \simeq G'.$ Since $G$ is BFC-group, we have $G'$ is finite  \cite{BHN}. Consequently, $\mu(G)$ is a subgroup of finite index in $\Upsilon(G)$. By Lemma \ref{lem.general} (iii), the derived subgroup of $\nu(G)$ is given by $$\nu(G)' = \Upsilon(G) G' (G')^{\varphi}.$$ Since $G'$ is finite, we conclude that $\mu(G)$ is a central subgroup of finite index in $\nu(G)'$. \\

{\noindent} {$(b)$ implies $(c)$.} By Schur's theorem \cite[10.1.4]{Rob}, $\nu(G)''$ is finite. \\

{\noindent} {$(c)$ implies $(a)$.} By Lemma \ref{lem.general} (iii), $\nu(G)'' = [G',(G')^{\varphi}] G'' (G'')^{\varphi}$. Thus $[G',(G')^{\varphi}]$ is finite. Since $G'$ is finitely generated, it follows that $G'$ is finite (Lemma \ref{fg}). According to Neumann's result $G$ is BFC-group, which completes the proof.  
\end{proof}

\begin{rem} \label{dihedral}
In Theorem~B, the hypothesis that the derived subgroup $G'$ is finitely generated is essential. Let $p \geqslant 3$ be a prime and the Pr\"uffer group $A = \mathbb{Z}(p^{\infty})$. We define the semi-direct product $D = A \cdot C_2$, where $C_2 = \langle c \rangle$ and $$a^c = -a,$$ for every $a \in A$. Thus $D' \cong A$. By Lemma \ref{lem.general} (iii), $$\nu(D)'' = \Upsilon(D') D'' (D'')^{\varphi}.$$ 
Since $D'$ is a Pr\"uffer group, it follows that the non-abelian tensor square $\Upsilon(D')$ is trivial. As $D$ is metabelian group we have $\nu(D)''$ is trivial (and so, finite). On the other hand, $D$ is not BFC-group. 
\end{rem}

It is well know that the converse of Schur's theorem does not hold. On the other hand, the following result guarantees that $\nu$-groups satisfies the converse of Schur's theorem. 

\begin{cor}
Let $G$ be a $\nu$-group. If $G'$ is finite, then $G$ is central-by-finite.
\end{cor}

\begin{proof}
By definition, there exists a group $H$ with derived subgroup $H'$ finitely generated such that $G \cong \nu(H)'$. Applying Theorem~B to $\nu(H)''$, we obtain that $\nu(H)'$ is central-by-finite, as required. 
\end{proof}

\begin{cor} \label{cor.1}
Let $G$ be a BFC-group. Then $\Upsilon(G)$ is central-by-finite. Moreover, $\exp(\Upsilon(G)')$ divides $|G'|$. 
\end{cor}
\begin{proof}
Set $K = \Upsilon(G)$. Therefore $K$ is central-by-finite by Theorem~B (b). That $exp \ (K')$ divides $|K/Z(K)|$ follows from Schur's theorem \cite[10.1.4]{Rob}. Since $G$ is BFC-group and $K/\mu(G) \simeq G'$, we have $|K/Z(K)|$ divides $|G'|$, which completes the proof. 
\end{proof}

\begin{rem} \label{rem.d}
Let $G$ be a BFC-group. That $\Upsilon(G)$ is central-by-finite follows from Corollary \ref{cor.1}. Arguing as in the Introduction we deduce that $\Upsilon(D_{\infty})$ is abelian, where $D_{\infty} = \langle a,b \mid \ a^2, a^b = a^{-1}\rangle$. More precisely, $$\Upsilon(D_{\infty}) \cong C_2 \times C_2 \times C_2 \times C_{\infty}.$$ On the other hand, $D'_{\infty}$ is cyclic subgroup and $D_{\infty}$ is not BFC-group. 
\end{rem}

This suggest the following question. \\ 

\noindent {\it Let $G$ be a group in which $G'$ is generated by finitely many commutators of finite order and $\Upsilon(G)$ is central-by-finite. Is it true that $G$ is BFC-group?}

\end{document}